\documentclass[12pt]{amsart}
\usepackage[a4paper,twoside,left=2cm,right=2cm,top=3.1cm,bottom=2.3cm]{geometry}

\usepackage[latin1]{inputenc}
\usepackage{amssymb,amsmath,amsfonts, amsthm}
\usepackage{enumerate}
\usepackage{mathrsfs}
\usepackage{times}
\usepackage{xcolor}
\usepackage{cite}

\usepackage{amsmath,amssymb,amscd,amsthm}
\usepackage{latexsym}
\usepackage{graphicx}
\usepackage[latin1]{inputenc}       % allow Latin1 characters
\usepackage{textcomp}
\usepackage{times}
\usepackage{colortbl}

\newtheorem{thm}{Theorem}[section]

\newtheorem{theorem}{Theorem}[section]
\newtheorem{corollary}[theorem]{Corollary}
\newtheorem{proposition}[theorem]{Proposition}
\newtheorem{lemma}[theorem]{Lemma}

\newtheorem{definition}[theorem]{Definition}

\newtheorem{remark}[theorem]{Remark}
\newtheorem{example}[theorem]{Example}

\newtheorem{problem}[theorem]{Problem}

\newtheorem{thmx}{Theorem}

\def\irr#1{{\rm Irr}(#1)}
\def\irrr#1#2 {\irr {#1 \mid #2}}

\newcommand{\R}{\mathbb R}
\newcommand{\N}{\mathbb N}

\newcommand{\s}{\mathbb S}
\newcommand{\sfe}{{{\mathbb S}^{n}}}

\newcommand{\symm}{\mbox{\rm Sym}}

\newcommand{\M}{{\mathcal M}}

\begin{document}

\title[]{A constant rank theorem\\
for linear elliptic equations on the sphere\\
with applications to the mixed Christoffel problem}
\author[A. Colesanti, M. Focardi, P. Guan, P. Salani]{A. Colesanti, M. Focardi, P. Guan, P. Salani}
\address{Dipartimento di Matematica e Informatica ``U. Dini", Universit\`a degli Studi di Firenze}
\email{colesant@math.unifi.it}
\address{Dipartimento di Matematica e Informatica ``U. Dini", Universit\`a degli Studi di Firenze}
\email{matteo.focardi@unifi.it}
\address{Department of Math. \& Stats, McGill University}
\email{pengfei.guan@mcgill.ca}
\address{Dipartimento di Matematica e Informatica ``U. Dini", Universit\`a degli Studi di Firenze}
\email{paolo.salani@unifi.it}
\keywords{convex body; mixed are measures; mixed Christoffel problem; elliptic PDE; constant rank theorem}
\begin{abstract} 
We study the mixed Christoffel problem for $C^{2,+}$ convex bodies providing sufficient conditions for its solution. Key to our approach is a constant rank theorem, following the approach developed in \cite{Guan-Ma-2003} to address the Christoffel problem, in order to ensure that the solution to a related second order linear PDE on the sphere is indeed geometric, that is, it is the support functions of a $C^{2,+}$ convex body. 
\end{abstract} 

%\thanks{Partially supported by G.N.A.M.P.A (INdAM) and by the FIR project 2013: 
%{\em Geometrical and qualitative aspects of PDE's}.}
%\subjclass[2010]{26B25, 52A39} 
%\keywords{Log-concave functions; convex bodies; functional inequalities; valuations}
%\date{\today}
%\begin{abstract} \end{abstract} 
\maketitle

\section{Introduction}

In this paper we are interested in {\em (mixed) area measures} of {\em convex bodies} (compact and convex subsets of $\R^{n+1}$, with non-empty interior). Let us briefly explain how these measures are generated. It is well known that the volume (that is, the Lebesgue measure, here denoted by $V$) of the Minkowski linear combination 
$$
t_1\Omega_1+\dots+t_m\Omega_m,\quad t_1,\dots,t_m\ge0,
$$
of $m$ convex bodies in $\R^{n+1}$ is an homogeneous polynomial of degree $(n+1)$ in the variables $t_i$'s:
$$
V(t_1\Omega_1+\dots+t_m\Omega_m)=\sum_{i_1,\dots,i_{n+1}=1}^m V(\Omega_{i_1},\dots,\Omega_{i_{n+1}})t_{i_1}\cdots t_{i_{n+1}}.
$$
The coefficients of this polynomial, $V(\Omega_{i_1},\dots,\Omega_{i_{n+1}})$, are the so-called {\em mixed volumes}. Hence, each mixed volume is a quantity depending on $(n+1)$ convex bodies:
$$
V=V(\Omega_1,\dots,\Omega_{n+1}).
$$
Mixed volumes are among the fundamental notions in convex geometry (see for instance \cite{Schneider}). They include {\em quermassintegrals} as special cases: for a fixed convex body $\Omega$, mixed volumes of the form
$$
V(\underbrace{\Omega,\dots,\Omega}_{\mbox{\tiny $k$-times}},\underbrace{B_1,\dots,B_1}_{\mbox{\tiny $(n+1-k)$-times}})
$$
where $B_1$ is the unit ball of $\R^{n+1}$ and $k\in\{1,\dots,n+1\}$, are called {\em quermassintegrals of $\Omega$} (or, with a different normalization, its {\em intrinsic volumes}). 

Mixed area measures arise naturally in integral representations of mixed volumes. Given $n$ convex bodies $\Omega_1,\dots,\Omega_n$ in $\R^{n+1}$, there exists a finite Borel measure $S=S(\Omega_1,\dots,\Omega_n;\cdot)$, defined on the unit sphere $\sfe$ of $\R^{n+1}$, such that, for any other convex body $\Omega$,  
$$
V(\Omega_1,\dots,\Omega_n,\Omega)=\int_{\sfe}u_{\Omega}(x)dS(\Omega_1,\dots,\Omega_n;x),
$$
where $u$ is the {\em support function} of $\Omega$. We recall that the support function of $\Omega$ (with respect to the origin) is the function $u_\Omega: \ \sfe\to\R$ defined by
$$
u_\Omega(x)=\sup_{z\in\Omega}\langle x,z\rangle,
$$
where $\langle\cdot,\cdot\rangle$ is the standard scalar product in $\mathbb{R}^{n+1}$. The measure $S=S(\Omega_1,\dots,\Omega_n;\cdot)$ is called the {\em mixed area measure} of $\Omega_1,\dots,\Omega_n$. Due to translation invariance of mixed volumes, area measures verify the following condition:
\begin{equation}\label{comp-cond 0-0}
\int_{\s^n} x_j dS(\Omega_1,\dots,\Omega_n;x)=0, \quad \forall j=1,\cdots, n+1
\end{equation}
(here $x_1,\dots,x_{+1}$ are the coordinates in $\R^{n+1}$).

A special class of mixed area measures related to a convex body $\Omega$ are its {\em area measures}, $S_k=S_k(\Omega;\cdot)$, $k=1,\dots,n$, defined by
$$
S_k(\Omega;\cdot)=S(\underbrace{\Omega,\dots,\Omega}_{\mbox{\tiny $k$-times}},\underbrace{B_1,\dots,B_1}_{\mbox{\tiny $(n-k)$-times}};\cdot).
$$

In particular, the measure $S_n(\Omega;\cdot)$ has the following simple geometric interpretation: for every Borel set $\eta\subset\sfe$,
$$
S_n(\Omega;\eta)={\mathcal H}^{n-1}(\nu_\Omega^{-1}(\eta)),
$$
where $\mathcal{H}^{n-1}$ is the $(n-1)$-dimensional Hausdorff measure, and $\nu_\Omega$ is the Gauss map of $\Omega$. 

The celebrated {\em Minkowski problem} requires to find a convex body $\Omega$ with prescribed area measure $S_n(\Omega;\cdot)$. This problem is well-posed: simple necessary and sufficient conditions on the given measure, in order to have existence and uniqueness (up to translations) of the solution, are known, and a satisfactory regularity theory has now been developed, thanks to the contributions of many authors. The literature concerning this problem is vast; the interested reader is referred, for instance, to Chapter 7 in \cite{Schneider}, for an exhaustive account on this subject.

On the other hand, the problem of finding a convex body with given area measure of order one, that is $S_1(\Omega;\cdot)$, called the {\em Christoffel problem}, appears to be much more delicate.  Necessary and sufficient conditions for the existence of a solution were established by Firey \cite{Firey} and Berg \cite{Berg} (see also Chapter 8 in \cite{Schneider} and \cite{LWW}), but they are in general not easy to be checked in practice. The situation is even more complicated for the {\em Christoffel-Minkowski problem}, concerning the area measures of intermediate order, $S_k(\Omega;\cdot)$, $2\le k\le (n-1)$. 

\medskip

\subsection{The passage to PDE's} A general observation, which is crucial for the present paper, is that, under regularity assumptions on the involved convex bodies, mixed area measures admit a density with respect to the uniform measure on $\sfe$, and such densities can be explicitly written in terms of the derivatives of the support function. This allows, for instance, to reduce the Minkowski and the Christoffel problems to find a solution of a partial differential equation on the unit sphere, within the class of support functions of convex bodies.

We say that a convex body $\Omega$ is of class $C^{2,+}$ if its boundary is of class $C^2$, and the Gauss curvature is positive at every point of $\partial\Omega$. If $\Omega$ is of class $C^{2,+}$, then its support function $u$ belongs to $C^2(\sfe)$. We denote by $u_i$, $u_{ij}$, $i,j=1,\dots,n$, the first and second covariant derivatives of $u$, with respect to a local orthonormal frame on $\sfe$. 

\begin{definition}\label{W} For a convex body $\Omega\subset \mathbb R^{n+1}$, of class $C^{2,+}$, with support function $u$, we define the inverse Weingarten form of $\Omega$ as
\begin{equation}\label{w-def} 
W_\Omega(x)=(u_{ij}(x)+u(x)\delta_{ij})_{i,j=1,\dots,n},\quad x\in\sfe.
\end{equation} 
\end{definition}

Note that $W$ is independent of the choice of the point with respect to which the support function is defined (that is, is invariant with respect to translations of $\Omega$). Moreover, $\Omega$ is of class $C^{2,+}$ if and only if its inverse Weingarten form is positive definite at every point of $\sfe$; see, for instance, Chapter 2 in \cite{Schneider}.

\medskip

We will denote by $\mathscr{D}$ the mixed discriminant operator, that is, the operator generated by the polarization of the determinant with respect to the usual addition of matrices. 
$\mathscr{D}$ acts on $n$-tuples of $n\times n$ symmetric matrices. With this notation, if $\Omega_1,\dots,\Omega_n$ are convex bodies of class $C^{2,+}$, and $W_i$ is the inverse Weingarten form of $\Omega_i$, $i=1,\dots,n$, then their mixed area measure is absolutely continuous with respect to the $(n-1)$-dimensional Hausdorff measure restricted to $\sfe$, and its density is
\begin{equation}\label{density of mixed area measures}
dS(\Omega_1,\dots,\Omega_n;x)=\mathscr{D}(W_1,\cdots, W_n)dx.
\end{equation}
In particular
$$
dS_n(\Omega;x)=\det(W_\Omega(x))dx
$$
for every convex body $\Omega$ of class $C^{2,+}$. 

We now observe that the support function of the unit ball $B_1$ of $\R^{n+1}$ (restricted to $\sfe$) is the constant function $1$, and consequently 
$$
W_{B_1}\equiv I_n,
$$
where $I_n$ is the identity matrix of order $n$. Hence 
$$
dS_k(\Omega;x)=\mathscr{D}(\underbrace{W_\Omega(x),\dots, W_\Omega(x)}_{\mbox{\tiny $k$-times}},\underbrace{I_n,\dots,I_n}_{\mbox{\tiny $(n-k)$-times}})dx=\sigma_k(W_\Omega(x))dx,
$$
where $\sigma_k(W_\Omega)$ denotes the $k$-th principal symmetric function of the eigenvalues of $W_\Omega$.

As a consequence, the Minkowski problem in the smooth setting can be reduced to solve a partial differential equation on $\sfe$: if $f$ is the density of the given measure, the problem amounts to find a solution $u$ to the Monge-Amp\`ere equation
$$
\det(u_{ij}(x)+u\delta_{ij}(x))=f(x),
$$
such that its associated inverse Weingarten form is positive definite on $\sfe$:
\begin{equation}\label{geometric solution}
W(x)=(u_{ij}(x)+u\delta_{ij}(x))>0\quad\forall\, x\in\sfe.
\end{equation}
The latter condition ensures that $u$ is the support function of a convex body (of class $C^{2,+}$). We will refer to this condition saying that $u$ is a {\em geometric solution} of the relevant equation. 

Similarly, the Christoffel problem is equivalent to find a geometric solution of the linear equation
\begin{equation}\label{CF-eq} 
\mbox{tr}(W)=\Delta_{\mathbb S^n} u+nu =f,
\end{equation}
where $\Delta_{\sfe}$ denotes the spherical Laplacian. More generally, for an arbitrary $k\in\{1,\dots,n\}$, the Christoffel-Minkowski problem is reduced to find a geometric solution of the following {\em Hessian equation}
\begin{equation}\label{CM-eq}
\sigma_k(u_{ij}+u\delta_{ij})=f.
\end{equation}
Here, given a symmetric matrix $W$, $\sigma_k(W)$ denotes the $k$-th elementary symmetric function of the eigenvalues of $W$. In all cases, in view of \eqref{comp-cond 0-0}, the compatibility condition
\begin{equation}\label{eq condition in GM}
\int_{\mathbb S^n}x_j f(x)dx=0,\quad\forall\ j=1,\dots,n+1,    
\end{equation}
is a necessary condition on $f$. This condition is verified, for instance, when $f$ (or equivalently $\mu$) is even. In the case of Minkowski problem (in the smooth setting), \eqref{eq condition in GM} is also sufficient. On the other hand it is known that such condition is not sufficient for the existence of a solution of the Christoffel problem (see \cite[Section 8.3]{Schneider}). 

\subsection{A sufficient condition for the Christoffel problem} Existence of solution to \eqref{CF-eq} follows from the standard theory of elliptic PDE's. The main question is when such solution $u$ is geometric, i.e. $u$ is a support function of convex body. A progress was made in the paper \cite{Guan-Ma-2003}, where the authors established the following {\em constant rank} type theorem, for solutions of equation \eqref{CF-eq}, under a fairly simple condition on $f$ (in the same paper they proved a similar result for equation \eqref{CM-eq}). 

\begin{thmx}\label{GM thm}
Let $f\in C^{1,1}(\sfe)$ be a strictly positive function satisfying (\ref{eq condition in GM}).
Assume moreover that
\begin{equation}\label{GM condition}
\mbox{the $1$-homogeneous extension of $\dfrac 1f$ to $\R^{n+1}$, is convex.}    
\end{equation} 
If $u$ is a solution to \eqref{CF-eq} with 
$$
W(x)=(u_{ij}(x)+u(x)\delta_{ij})\ge 0\quad\forall\, x\in\sfe,
$$ 
then $W(x)>0$ for every $x\in\mathbb S^n$. 
\end{thmx}  

Using the previous result and standard continuity argument, existence of a geometric solution to \eqref{CF-eq} can be proved if there is a homotopic path, described by a real parameter $t\in[0,1]$, connecting $f$ and, say, the constant function $1$, such that the assumption (\ref{GM condition}) is valid for every $t$. For that purpose, Andrews and Ma used a curvature flow in an unpublished note \cite{AM}; see also a recent work by Bryan-Ivaki-Scheuer \cite{BIS}. A direct homotopic path was constructed by Sheng-Trudinger-Wang in \cite{STW}. Therefore, as a corollary of Theorem \ref{GM thm}, conditions \eqref{comp-cond 0-0} (that is \eqref{eq condition in GM}) and \eqref{GM condition} are sufficient for the existence of a solution of the Christoffel problem. Note that \eqref{comp-cond 0-0} is also a necessary condition.

\subsection{The mixed Christoffel problem} In this paper we study the following problem.

\begin{problem}[{\bf The mixed Christoffel problem}] Let $\Omega_1,\dots,\Omega_{n-1}$, be given convex bodies in $\R^{n+1}$, and let $\mu$ be a Borel measure on $\sfe$. Find a convex body $\Omega$ in $\R^{n+1}$, such that
\begin{equation}\label{MC problem}
S(\Omega_1,\dots\Omega_{n-1},\Omega;\cdot)=\mu(\cdot).
\end{equation}
\end{problem}

This problem is mentioned in the recent paper \cite{Huang-Yang-Zhang-2025}. Moreover, necessary and sufficient conditions for the existence of a solution in the rotation invariant case have been recently established in \cite{Brauner-Hofstatter-Ortega1, Brauner-Hofstatter-Ortega2, Mussnig-Ulivelli}.

If $\Omega_i=B_1$ for every $i=1,\dots,(n-1)$, then the problem reduces to the Christoffel problem. When $L$ is a given convex $C^{2,+}$ set (different from $B_1$) and $\Omega_i=L$ for every $i=1,\dots,(n-1)$, then the problem can be interpreted as an {\em anisotropic Christoffel problem}, where $L$ plays the role of the unit ball in place of $B_1$. 

We first note that, due to \eqref{comp-cond 0-0}, the following is a necessary condition for the existence of a solution:
\begin{equation}\label{comp-cond 0}
\int_{\s^n} x_j d\mu(x)=0, \quad \forall j=1,\cdots, n+1.
\end{equation}

We will treat the mixed Christoffel problem in the smooth setting: we will assume that $\Omega_1,\cdots, \Omega_{n-1}$ are convex bodies of class $C^{2,+}$, and that the measure $\mu$ has a smooth density $f$ with respect to the $n$-dimensional Hausdorff measure restricted to $\sfe$. In particular, condition \eqref{comp-cond 0} is equivalently expressed by \eqref{eq condition in GM}.

We establish a sufficient condition involving $\Omega_1,\dots,\Omega_{n-1}$ and $f$, for the existence of a solution to the mixed Christoffel problem. 

In order to state our existence result, we need to recall the notion of mixed co-factor matrix. Let $M_1,\dots,M_{n-1}$ be $n\times n$ symmetric matrices. For any further $n\times n$ symmetric matrix $M=(m_{ij})_{i,j=1\dots,n}$, the mixed discriminant
$$
\mathscr{D}(M_1,\dots,M_{n-1},M)
$$
is linear with respect to $M$. Therefore it can be written in the form
$$
\mathscr{D}(M_1,\dots,M_{n-1},M)=\sum_{i,j=1}^n c_{ij}m_{ij},
$$
for suitable and uniquely determined coefficients $c_{ij}\in\R$. The matrix
$$
C(M_1,\dots,M_{n-1})=(c_{ij})_{i,j=1,\dots,n},
$$
is called the mixed co-factor matrix of $M_1,\dots, M_{n-1}$. Alternatively, the coefficients $c_{ij}$ can be defined as
\begin{equation}\label{cmc}
c_{ij}=\frac{\partial \mathscr{D}(M_1,\dots,M_{n-1},M)}{\partial m_{ij}},\quad\forall\, i,j=1,\dots,n.
\end{equation}

By \eqref{density of mixed area measures}, condition \eqref{MC problem} is equivalent to
\begin{equation}\label{Prob-C}
\mathscr{D}(W_1,\cdots, W_{n-1}, W)=f, \quad \mbox{on $\s^n$},
\end{equation}
and the latter equation can be written in the form:
\begin{equation}\label{the equation}
\sum_{\alpha,\beta=1}^n a_{\alpha\beta}(x)(u_{\alpha\beta}(x)+\delta_{\alpha\beta}u(x))=f(x)
\end{equation}
(once that a local orthonormal coordinate system has been fixed on $\sfe$), where 
$$
(a_{\alpha\beta})_{\alpha,\beta=1,\dots,n}
$$ 
is the mixed co-factor matrix of $W_1,\cdots, W_{n-1}$. Equation \eqref{the equation} is a linear second order equation in the unknown function $u$. The fact that $\Omega_i$ is of class $C^{2,+}$ for every $i$, implies that this equation is uniformly elliptic, which is equivalent to
\begin{equation}\label{ellipticity}
(a_{\alpha\beta}(x))>0,\quad\forall\ x\in\sfe.
\end{equation}

The previous conditions guarantee the existence of a solution to (\ref{the equation}). Moreover, the solution is unique modulus the addition of linear functions (see, for instance, \cite{GMTZ}). We also observe that if $\partial \Omega_1,\cdots, \partial \Omega_{n-1}\in C^{4}$, then $a_{\alpha\beta}\in C^{2}$. Under these assumptions, the admissible solution $u$ of (\ref{the equation}) belongs to $C^{4,\gamma}$ (but note that $C^{3,1}$ would be sufficient for us to carry out convexity estimates in later sections).

We look for a solution $u$ of \eqref{the equation} such that condition \eqref{geometric solution} holds. As in \cite{Guan-Ma-2003}, we will prove a constant rank theorem for \eqref{the equation}. In order to state our result, we need the following remark. 

\begin{remark} Assume that $A$ is a symmetric $(2,0)$ tensor defined on an open subset $D$ of $\sfe$, and that $u\in C^2(D)$. For a fixed local coordinate system on $\sfe$, let 
$$
A=(a_{ij})_{i,j=1,\dots,n},\quad W=(w_{ij})_{i,j=1,\dots,n}=(u_{ij}+u\delta_{ij})_{i,j=1,\dots,n}.
$$
Then the quantity
$$
{\rm tr}(AW)=\sum_{i,j=1}^n a_{ij}w_{ij},
$$
is independent of the local coordinate system chosen on $\sfe$. 
\end{remark}

In the following statement $(a^{\alpha\beta})_{\alpha,\beta=1,\dots,n}$ denotes the inverse matrix of $(a_{\alpha\beta})_{\alpha,\beta=1,\dots,n}$.

\begin{thm}\label{main theorem} Let $A(x)=(a_{ij}(x))\in C^{2,\gamma}(\s^n)$ be a positive $(2,0)$-tensor in an open and connected set $D\subset \s^n$ and let $f\in C^{2, \gamma}(D)$ for some $\gamma>0$, with $f>0$ on $\s^n$. Suppose that, for every $1\le l\le n-1$, and every $q=1,\dots,n$,
\begin{equation}\label{cond-l} 
\left(\delta_{ij}+ \frac{\nabla_i \nabla_j (\frac{a_{qq}}{f})}{\frac{a_{qq}}{f}}+\frac12\frac{\nabla_ia_{qq}\nabla_ja_{qq}}{a^2_{qq}}-\frac12\sum_{\alpha,\beta=1}^l \frac{a^{\alpha\beta}\nabla_ia_{q\alpha}\nabla_ja_{q\beta}}{a_{qq}}\right)_{i,j=1,\dots,n}\ge 0.
\end{equation}
Let $u\in C^2(D)$ be a solution of equation \eqref{the equation} in $D$. If 
\[
W=(u_{ij}+u\delta_{ij})_{i,j=1,\dots,n}\ge 0
\] 
in $D$, then $W$ has constant rank in $D$. Moreover, if $D=\s^n$ (and $n\ge2$), then $W$ has maximal rank, that is $W>0$ in $D$.

In the case $n=2$, condition (\ref{cond-l}) can be written as follows: for every $q=1,2$,
\begin{equation}\label{cond-1 2} 
\left(\frac{a_{qq}}f\delta_{ij}+ \nabla_i\nabla_j\left(\frac{a_{qq}}{f}\right)\right)_{i,j=1,2}\ge 0.
\end{equation}
\end{thm}

Thank to \eqref{cond-1 2}, for $n=2$ a simplified version of the previous result can be achieved. We need to introduce a definition, first. For an arbitrary $k\in\N$, we denote by $\symm(k)$ the space of square and symmetric matrices of order $k$.

\begin{definition} Let $n,k\in\N$, and let
$$
M\colon\R^{n+1}\to\symm(k).
$$
We say that the function $M$ is convex if
$$
M((1-t)x_0+tx_1)\le(1-t)M(x_0)+tM(x_1),
$$
for every $x_0, x_1\in\R^{n+1}$ and for every $t\in[0,1]$. The above inequality has to be intended in the matricial sense: the difference between the right hand side and the left hand side is a positive semi-definite matrix. 
\end{definition}

\begin{thm}\label{main theorem n=2 rephrased} Let $A(x)=(a_{ij}(x))_{i,j=1,2}\in C^{2,\gamma}(D)$ be a positive $(2,0)$-tensor defined in an open connected set $D\subset\s^2$, and let $f\in C^{2, \gamma}(D)$ for some $\gamma>0$, with $f>0$ on $D$. Let $M$ be the $1$-homogeneous extension of the tensor
$$
\frac{A}{f}
$$
to the cone 
$$
C=\left\{x\in\R^3\colon\frac{x}{|x|}\in D\right\}.
$$
Assume that  
\begin{equation}\label{the condition}
M\mbox{ is locally convex in $C$}.
\end{equation}
Let $u\in C^2(\s^2)$ be a solution of equation \eqref{the equation}. If  
\[
W=(u_{ij}+u\delta_{ij})_{i,j=1,2}\ge 0
\] 
in $D$, then $W$ has constant rank in $D$. In particular, if $D=\s^2$ then $W$ has full rank in $\s^2$. 
\end{thm}

\medskip

Using a continuity argument, from Theorem \ref {main theorem} it can be deduced that the solution of equation \eqref{the equation} is in fact a geometric solution, that is, its $1$-homogeneous extension to $\R^n$ is convex. The application of the continuity argument requires additional assumptions, which are more restrictive in the case $n\ge3$, due to the non-linearity of condition \eqref{cond-l} with respect to the matrix $A$. For this reason we present the result in two distinct statements.

\begin{corollary}\label{the corollary} 
Let $n=2$. Under the same assumptions of Theorem \ref{main theorem n=2 rephrased} on $A$ and $f$, with $D=\s^2$, assume moreover that $f$ verifies \eqref{GM condition}. Let $u$ be a solution of equation \eqref{the equation}. Then 
\[
W=(u_{ij}+u\delta_{ij})>0\quad\mbox{on $\s^2$.}
\]  
In particular, the $1$-homogeneous extension of $u$ to $\R^n$ is convex.
\end{corollary}

\begin{corollary}\label{the corollary nge3} 
Let $n\ge3$. Under the same assumptions of Theorem \ref{main theorem} on $A$ and $f$, with $D=\sfe$, assume moreover that condition \eqref{cond-l} is verified by the pair $f$ and $A_t$, for every $t\in[0,1]$, where
$$
A_t:=(1-t)I_{n}+tA;
$$
(here $I_n$ is the identity matrix of order $n$). Let $u$ be a solution of equation \eqref{the equation}. Then 
\[
W=(u_{ij}+u\delta_{ij})>0\quad\mbox{on $\sfe$.}
\]  
\end{corollary}

As an application of the previous corollaries we get sufficient conditions for the existence of a solution to the mixed Minkowski problem. In view of the difference between the cases $n=2$ and $n\ge3$, we state them separately. In particular, the result in the case $n=2$ is stated in \S \ref{Mp2}. 

\begin{thm}\label{Thm MCP} Let $n\ge3$. Let $\Omega_1,\dots,\Omega_n$ be convex bodies in $\R^{n+1}$, with boundary of class $C^{4,\gamma}$, for some $\gamma>0$, and positive Gauss curvature. For every $r=1,\dots,n$, let $u_r$ be the support function of $\Omega_r$ and 
$$
W_r(x)=((u_r)_{ij}(x)+u_r(x)\delta_{ij})_{i,j=1\dots,n},\quad x\in\sfe,
$$
be the inverse Weingarten form of $\Omega_r$. Let 
$$
A(x)=(a_{ij}(x))
$$
be the mixed co-factor matrix of $W_1(x),\dots,W_{n-1}(x)$. Moreover, let
$$
A_t:=(1-t)I_{n}+tA,\quad t\in[0,1]
$$
Let $f\in C^{2, \gamma}(\s^n)$, with $f>0$ on $\s^n$, and assume that $f$ verifies condition \eqref{eq condition in GM}. Let $\mu$ be the Borel measure on $\sfe$ with density $f$ with respect to the $(n-1)$-dimensional Hausdorff measure restricted to $\sfe$. 

Suppose that, for every $1\le l\le n-1$, for every $t\in[0,1]$ and for every $q=1,\ldots,n$, the condition
\begin{equation}\label{cond-l 1} 
\left(\delta_{ij}+ \frac{\nabla_i \nabla_j (\frac{a_{qq}}{f})}{\frac{a_{qq}}{f}}+\frac12\frac{\nabla_ia_{qq}\nabla_ja_{qq}}{a^2_{qq}}-\frac12\sum_{\alpha,\beta=1}^l \frac{a^{\alpha\beta}\nabla_ia_{q\alpha}\nabla_ja_{q\beta}}{a_{qq}}\right)_{i,j=1,\dots,n}\ge 0
\end{equation}
is verified by the matrix $A_t$, for every $t\in[0,1]$.

Then there exists a convex body $\Omega$ in $\R^{n+1}$, of class $C^{2,+}$, such that
$$
S(\Omega_1,\dots,\Omega_{n-1},\Omega;\cdot)=\mu(\cdot).
$$
Moreover, $\Omega$ is uniquely determined up to translations.
\end{thm}

\begin{remark}\label{remark} If $A(x)=(a_{ij}(x))$ is the identity matrix for every $x\in\sfe$, that is, we are in the case of the ordinary Christoffel problem, condition \eqref{cond-l 1} becomes
$$
\left(\left(\frac1{f}\right)_{ij}+\delta_{ij}\frac{1}{f}\right)_{ij}\ge0.
$$
In other words, we get the sufficient condition \eqref{GM condition} found by Guan and Ma in \cite{Guan-Ma-2003}. 
\end{remark}

\begin{remark} Condition \eqref{GM condition} is obviously verified if $f\equiv 1$.
\end{remark}

\subsection{The mixed Christoffel problem in the $3$-dimensional space}\label{Mp2} The $3$-dimensional case ($n=2$), deserves a special treatment. The mixed Christoffel problem in $\R^3$ can be formulated as follows: let $\Omega_0$ be a convex body in $\R^3$, and let $\mu$ be a Borel measure on $\s^2$, such that
$$
\int_{\s^2}xd\mu(x)=0.
$$
The problem requires to find a convex body $\Omega$ such that
$$
S(\Omega_0,\Omega;\cdot)=\mu(\cdot).
$$
Assume that $\Omega_0$ is of class $C^{2,+}$, with support function $h\in C^2(\s^2)$; let 
$$
W_0=W_{\Omega_0}=(h_{ij}+h\delta_{ij})_{i,j=1,2},
$$
be the inverse Weingarten map of $\Omega_0$. Assume moreover that $\mu$ has a density $f\in C^{2,\gamma}(\s^2)$, for some $\gamma\in(0,1)$. 

Using Theorem \ref{main theorem n=2 rephrased}, or equivalently \eqref{cond-1 2}, we will derive two sufficient conditions for the existence of a solution, stated in the following Theorems \ref{MC 3dim 1} and \ref{MC 3dim}.  

\begin{thm}\label{MC 3dim 1} Let $\Omega_0\subset\R^3$ be a convex body with boundary of class $C^{4,\gamma}$, for some $\gamma>0$, and with positive Gauss curvature. Let $h$ be the support function of $\Omega_0$. Let $f\in C^{2,\gamma}(\s^2)$, with $f>0$ on $\s^2$, be such that condition \eqref{GM condition} holds. Assume that, for every $x\in\s^2$, and for every tangential direction $\alpha$ to $\s^2$ at $x$, the following inequality holds: 
\begin{equation}\label{new form dim two statement}
\left(\frac{h_{\alpha\alpha}(x)+h(x)}{f(x)}\right)_{\alpha\alpha}+\frac{h_{\alpha\alpha}(x)+h(x)}{f(x)}\ge0,
%\quad\forall\, x\in\s^2,\, \forall\, \alpha,
\end{equation}
where sub-indices denote covariant partial derivatives with respect to the direction $\alpha$. Then, there exists a convex body $\Omega\in C^{2,+}$ such that
$$
S(\Omega_0,\Omega,\cdot)=\mu(\cdot)\,,
$$
where $\mu$ is the measure such that $d\mu(x)=f(x)dx$.
\end{thm}

\begin{example} Let us show an example of application of the previous result. Let $f\equiv1$ so that, in particular, condition \eqref{GM condition} is verified. In this case \eqref{new form dim two statement} becomes:
\begin{equation}\label{e:new form two d statement bis}
h_{\alpha\alpha\alpha\alpha}+2h_{\alpha\alpha}+h>0, \quad\mbox{on $\s^2$.}
\end{equation}
Assume that $h\colon\sfe\to\R$ is of the form
$$
h=C+\psi
$$
where $C>0$ and $\psi\in C^4(\s^2)$ verifies
$$
\|\psi\|_{C^4(\s^2)}<\frac C4.
$$
Then \eqref{e:new form two d statement bis} is verified. 
\end{example}

\begin{thm}\label{MC 3dim} Let $\Omega_0\subset\R^3$ and $f$ be as in Theorem \ref{MC 3dim 1}, and let $W_0$ be the inverse Weingarten map of $\Omega_0$. If $f$ verifies condition \eqref{GM condition}, and the $1$-homogeneous extension of the tensor 
$$
\frac{W_0}{f}
$$
is convex in $\R^3$, then there exists a convex body $\Omega\in C^{2,+}$ such that
$$
S(\Omega_0,\Omega,\cdot)=\mu(\cdot),
$$
where $\mu$ is the Borel measure on $\s^2$ such that $d\mu=f dx$.
\end{thm}

\begin{remark} In the assumptions of Theorem \ref{MC 3dim}, let $h$ be the support function of $\Omega_0$, extended as a $1$-homogeneous function to $\R^3$. Let $D^2 h(x)$ denote the (Euclidean) Hessian matrix of $h$, for every $x\ne0$. $D^2h$ is a matrix-valued function defined in $\R^n\setminus{0}$, which extends the tensor $W_0$. Moreover, as $h$ is $1$-homogeneous, $D^2h$ is $-1$ homogeneous. As a consequence
$$
x\,\rightarrow\,|x|^2 D^2h(x),
$$
is a $1$-homogeneous extension of $W_0$. Hence 
$$
x\,\rightarrow\,|x|^2 \left[f\left(\frac x{|x|}\right)\right]^{-1}D^2h(x),
$$
is the $1$-homogeneous extension of 
$$
\frac{W_0}{f}.
$$
Therefore, the sufficient condition of Theorem \ref{MC 3dim}, for the existence of a solution to the mixed Christoffel problem is equivalent to the convexity of the matrix-valued map $M\colon\R^n\to{\rm Sym}(n+1)$ defined by
$$
M(x)=|x|^2 \left[f\left(\frac x{|x|}\right)\right]^{-1}D^2h(x).
$$
\end{remark}

The rest of the paper is organized as follows. 
Section \ref{section n=2} is devoted to the proofs of all the results established in the three-dimensional case ($n=2$).

In Section \ref{section n general} we consider the case $n\ge3$.

The proof of Corollaries \ref{the corollary} and \ref{the corollary nge3} is given in Section \ref{final section}.

\medskip

\noindent{\bf Acknowledgments.} A. Colesanti was supported by the project {\em Disuguaglianze analitiche e geome\-triche}, funded by the Gruppo per l'Analisi Matematica la Probabilit\`a e le loro Applicazioni. The research of the P. Guan is supported in part by an NSERC Discovery grant. M. Focardi has been supported by the European Union - Next Generation EU, Mission 4 Component 1 CUP B53D2300930006, codice 
2022J4FYNJ, PRIN2022 project ``Variational methods for stationary and evolution problems with singularities and interfaces''. P. Salani has been supported by the European Union - Next Generation EU, Mission 4 Component C2, PRIN2022 project ``Geometric-Analytic Methods for PDEs and Applications (GAMPA)'',  codice 
2022SLTHCE, CUP B53D23009460006. M. Focardi and P, Salani are members of GNAMPA - INdAM.

The project was initiated in 2008 and at the time we were only be able to deal with $n=2$ case. The general dimensional case was worked out very recently. P. Guan author would like to thank the department of Mathematics at the University of Florence for the hospitality.

\section{The case $n=2$}\label{sec The case n=2}

\subsection{Proof of Theorem \ref{main theorem} in the case $n=2$}\label{section n=2} 

The proof is based on the constant rank method. 
We start with an observation which will be used also in the general case: using equation \eqref{the equation}, the fact that the tensor $A$ is positive definite, and the condition $f>0$, we deduce that the rank of $W$ can not be $0$. 
In particular, as $n=2$, the rank of $W$ can be $1$ or $2$. 

If the rank equals $2$ at every point, there is nothing to prove. 

Hence, to proceed we assume that there exists a point $x_0$ where the rank is $1$. In a sufficiently small neighborhood $\mathcal U$ of $x_0$, one of the eigenvalues of $W$ is larger than an absolute positive constant $C_0$, while the other one is close to zero, and can be assumed to belong to $[0,C_0/2]$, that is,
$$
0\leq\lambda_1\leq\frac{C_0}{2}<C_0\leq\lambda_2\quad\text{in }\mathcal U\,,
$$
where, for $z\in\mathcal U$, $\lambda_1,\lambda_2$ denotes the eigenvalues of $W(z)$.  Let us notice that, by the standard regularity theory for elliptic equations, the solution $u\in C^{4,\gamma}(D)$ for some $\gamma\in(0,1)$. Then $\det(W), \lambda_1$ and $\lambda_2$ are of class $C^{2}({\mathcal U})$. 
We set
\begin{equation}\label{test function two dim}
\phi(x)=\det(W(x))=w_{11}(x)w_{22}(x)-w_{12}^2(x),\quad x\in{\mathcal U}.
\end{equation}

We now fix $z\in{\mathcal U}$ and choose a local orthonormal frame of $\s^2$, such that $W(z)$ is diagonal: note that this is possible since, by our assumptions, the quantity
$$
\sum_{i,j=1}^2 a_{ij}w_{ij}=\sum_{i,j=1}^2 a_{ij}(u_{ij}+u\delta_{ij})
$$
is independent of the local coordinate system. In the next relations, all the involved functions will be computed at $z$. From the definition of $\phi$ we have
$$
w_{11}=\frac1{w_{22}}\phi,
$$
so that
\begin{equation}\label{p1}
0\leq w_{11}\le \frac{1}{C_0}\phi.
\end{equation}
In the sequel we will denote by $C$ a general constant independent of $z\in\mathcal U$, and depending possibly only on $\|u\|_{C^3({\mathcal U})}$, $\|\phi\|_{C^2({\mathcal U})}$, {$\|A\|_{C^2({\mathcal U})}$} and $C_0$ (the value of $C$ may change even in the same line). 

Differentiating \eqref{test function two dim} with respect to $x_1$, we obtain
$$
w_{111}w_{22}=\phi_1-w_{11}w_{221}+2w_{12}w_{112}=\phi_1-w_{11}w_{22}\frac{w_{221}}{w_{22}}=\phi_1-\phi\frac{w_{221}}{w_{22}}.
$$
Whence, using again $w_{22}=\lambda_2\geq C_0$,
$$
|w_{111}|\le C(\phi+|\nabla\phi|).
$$
A similar inequality can be deduced for $w_{112}$. Therefore
\begin{equation}\label{p2}
|\nabla w_{11}|\le C (\phi+|\nabla \phi|).
\end{equation}
Our aim is to compute
$$
\sum_{\alpha,\beta=1}^2a_{\alpha\beta}(z)\phi_{\alpha\beta}(z).
$$
We start observing that, as a consequence of the Codazzi formula, the quantity
\begin{equation}\label{Codazzi}
w_{ijk},\quad i,j,k=1,\dots,n,
\end{equation}
is independent of the order of $i,j,k$. We have, for $\alpha,\beta=1,2$,
\begin{eqnarray*}
\phi_{\alpha\beta}&=&w_{11\alpha\beta}w_{22}+w_{11\alpha}w_{22\beta}+w_{11\beta}w_{22\alpha}+w_{11}w_{22\alpha\beta}-2(w_{12\alpha}w_{12\beta}+w_{12}w_{12\alpha\beta})
%\\  
%&=&w_{22}\left(w_{11\alpha\beta}-2\frac1{w_{22}}w_{12\alpha}w_{12\beta}\right)+O(\phi+|\nabla\phi|).
\end{eqnarray*}
Taking into account \eqref{p1} and \eqref{p2}, we deduce the following expression for the second order derivatives of $\phi$:
\begin{equation}\label{general expression}
\phi_{\alpha\beta}=w_{22}\left(w_{11\alpha\beta}-2\frac1{w_{22}}w_{12\alpha}w_{12\beta}\right)+ O(\phi+|\nabla \phi|),\quad\alpha,\beta=1,2.  
\end{equation}
Here and throughout the paper, by $O(\phi+|\nabla \phi|)$ we mean a quantity which is bounded in absolute value by a constant $C$ (as above) times  $(\phi+|\nabla \phi|)$.

We also have the identities (see for instance \cite{Guan-Ma-2003}):
\begin{equation}\label{covariant 2}
w_{ii\alpha\beta}=w_{\alpha\beta ii}+w_{ii}-w_{\alpha\beta},\quad\alpha,\beta,i=1,2,  
\end{equation}
so that
$$
w_{11\alpha\beta}=w_{\alpha\beta11}-w_{\alpha\beta}+w_{11}=w_{\alpha\beta11}-w_{\alpha\beta}+O(\phi+|\nabla \phi|).
$$
Hence
$$
\phi_{\alpha\beta}= w_{22}w_{\alpha\beta11}-2w_{12\alpha}w_{12\beta}-w_{22}w_{\alpha\beta}+O (\phi+|\nabla \phi|).
$$
Next we compute
$$
\sum_{\alpha,\beta=1}^2a_{\alpha\beta}\phi_{\alpha\beta}= w_{22}\sum_{\alpha,\beta=1}^2a_{\alpha\beta}w_{\alpha\beta11}-2a_{22}w_{221}^2-w_{22}\sum_{\alpha,\beta=1}^2a_{\alpha\beta}w_{\alpha\beta}+O(\phi+|\nabla \phi|),
$$
or, equivalently,
\begin{equation}\label{formula}
-\frac1{w_{22}}\sum_{\alpha,\beta=1}^2a_{\alpha\beta}\phi_{\alpha\beta}=-\sum_{\alpha,\beta=1}^2a_{\alpha\beta}w_{\alpha\beta11}+2\frac{a_{22}w_{221}^2}{w_{22}}+a_{22}w_{22}+O(\phi+|\nabla \phi|).
\end{equation}
Here, we have used \eqref{p1} to absorb the term $a_{11}w_{11}$ into $O(\phi+|\nabla \phi|)$.

Next, we differentiate twice equation \eqref{the equation} with respect to $x_1$, and we get:
\begin{eqnarray*}
\sum_{\alpha,\beta=1}^2a_{\alpha\beta}w_{\alpha\beta11}&=&f_{11}-2\sum_{\alpha,\beta=1}^2a_{\alpha\beta1}w_{\alpha\beta1}-\sum_{\alpha,\beta=1}^2a_{\alpha\beta11}w_{\alpha\beta}\\
&=&f_{11}-2a_{221}w_{221}-a_{2211}w_{22}+O(\phi+|\nabla \phi|),
\end{eqnarray*}
where $a_{\alpha\beta1}$ and $a_{\alpha\beta11}$ denote the first and the second partial derivatives of $a_{\alpha\beta}$ with respect to $x_1$, respectively, and we used \eqref{p1} and \eqref{p2}. Replacing the last equality in \eqref{formula}  we obtain:
\begin{eqnarray*}
-\frac1{w_{22}}\sum_{\alpha,\beta=1}^2a_{\alpha\beta}\phi_{\alpha\beta}&=&-f_{11}+2a_{221}w_{221}+2a_{2211}w_{22}+2\frac{a_{22}w_{221}^2}{w_{22}}+a_{22}w_{22}+O(\phi+|\nabla \phi|)\\
&=&f-f_{11}+2\frac{w_{221}}{w_{22}}(a_{22}w_{22})_1+a_{2211}w_{22}+O(\phi+|\nabla \phi|).
\end{eqnarray*}
Using again the equation (\ref{the equation}), we obtain
$$
w_{22}=\frac{f}{a_{22}}-\frac{a_{11}}{a_{22}}w_{11}=
\frac{f}{a_{22}}+O(\phi+|\nabla \phi|).
$$
Here we used \eqref{p1} and the fact that, due to the positivity of the tensor $A$, $a_{22}\ge c$ for some constant $c>0$ depending on $A$ only. Moreover, differentiating once \eqref{the equation} with respect to $x_1$ and taking into account \eqref{p1} and \eqref{p2}, we get
$$
a_{221}w_{22}+a_{22}w_{221}= f_1 +O(\phi+|\nabla \phi|).
$$
Therefore
$$
\frac{w_{221}}{w_{22}}=\frac{f_1}{f}-\frac{a_{221}}{a_{22}}+O(\phi+|\nabla \phi|).
$$
For simplicity, let $a=a_{22}$, $a_1=a_{221}$ and $a_{11}=a_{2211}$. From the previous steps we deduce
\begin{eqnarray*}
-\frac1{w_{22}}\sum_{\alpha,\beta=1}^2a_{\alpha\beta}\phi_{\alpha\beta}&
=&f-f_{11}+2\frac{f_1^2}{f}-2\frac{a_1f_1}{a}+\frac{fa_{11}}{a}+O(\phi+|\nabla \phi|)\\
&=&\frac{f^2}{a}\left[\left(\frac{a}{f}\right)_{11}+\frac af\right]+O(\phi+|\nabla \phi|).
\end{eqnarray*}
By \eqref{cond-1 2}, the first term in the right hand side of the previous relation is non-negative: 
\begin{equation}\label{added 1}
\left(\frac{a}{f}\right)_{11}+\frac af\ge0.
\end{equation}
Hence, we have the following differential inequality in $\mathcal U$:
\[
\sum_{\alpha,\beta=1}^2a_{\alpha\beta}\phi_{\alpha\beta}\le C (\phi+|\nabla \phi|),
\]
for every $z\in{\mathcal U}$. As $\phi\ge0$ and $\phi(x_0)=0$, it follows from the strong minimum principle that $\phi\equiv0$ in $D$, as $D$ is connected. This is equivalent to say that the rank of $W$ is constantly $1$ in $D$.

The proof of the final part of the statement, relative to the case $D=\s^2$, is provided in Section 3, in the more general setting of arbitrary dimension $n$.

\subsection{Proof of Theorems \ref{main theorem n=2 rephrased}, \ref{MC 3dim 1} and \ref{MC 3dim}}\label{sec. MC problem dim three}

We will need the following two remarks.

\begin{remark}\label{diagonal}
Let
$$
M\colon\R^{n+1}\to\symm(n),
$$
and let $M_{ij}$ denote the $ij$-th entry of $M$, for $i,j=1,\dots,n$. If $M$ is convex, then in particular $M_{ii}$ is a convex function in $\R^{n+1}$, for every $i=1,\dots,n+1$.  
\end{remark}

\begin{remark}\label{one-hom}
Let $g\colon\sfe\to\R$ and let $G\colon\R^{n+1}\to\R$ be the $1$-homogeneous extension of $g$. Assume that $g\in C^2(\sfe)$, and let $g_{ij}$, $i,j=1,\dots,n$, denote the second covariant derivatives of $g$, with respect to a local orthonormal frame $\{\bar e_1,\dots,\bar e_n\}$ on $\sfe$, defined on some open neighborhood of a point $x_0\in\sfe$. We now choose an orthonormal frame $\{e_1,\dots,e_n,e_{n+1}\}$ in $\R^{n+1}$, such that $\{e_1,\dots,e_n\}$ coincides with $\{\bar e_1,\dots,\bar e_n\}$ at $x_0$. Then
$$
g_{ij}(x_0)+g(x_0)\delta_{ij}=G_{ij}(x_0)\quad\forall\, i,j=1,\dots,n,
$$
where $G_{ij}$, $i,j=1,\dots,n$, denote the Euclidean second derivatives of $G$. A proof of this fact can be found, for instance, in \cite[Appendix]{Colesanti-Hug-Saorin-2014}.
\end{remark}

\begin{proof}[Proof of Theorem \ref{main theorem n=2 rephrased}] We need to prove that condition \eqref{the condition} implies \eqref{cond-1 2}. Let us fix an orthonormal frame on $\sfe$, and $q\in\{1,\dots,n\}$. By Remark \ref{diagonal}, the convexity of $M$ implies that the $1$-homogeneous extension of $(a_{qq}/f)$ to $\R^{3}$ is convex. Then \eqref{cond-1 2} follows from Remark \ref{one-hom}. 
\end{proof}

\begin{proof}[Proof of Theorem \ref{MC 3dim}] 
By \eqref{Prob-C}, the tensor $A=(a_{ij})_{i,j=1,2}$ appearing in \eqref{the equation} is given by
$$
A=C(W_0),
$$
where $C(W_0)$ is the co-factor matrix of $W_0$. Let
$$
W_0=
\left(
\begin{array}{cc}
b_{11}     &b_{12}  \\
b_{21}     &b_{22} 
\end{array}
\right).
$$
Then, by \eqref{cmc} we see that:
$$
(a_{ij})_{i,j=1,2}=
\left(
\begin{array}{cc}
b_{22}      &-b_{12}  \\
-b_{21}     &b_{11} 
\end{array}
\right)
$$
In particular, in this case the matrix 
$$
\frac1f(a_{ij})_{i,j=1,2}
$$ 
has the same diagonal entries of the matrix 
$$
\frac1f W_0,
$$ 
up to a permutation. The convexity of the $1$-homogeneous extension of the latter matrix then implies that assumption 
\eqref{cond-1 2} is verified, by the same argument used in the proof of Theorem \ref{main theorem n=2 rephrased}.
\end{proof}

\begin{proof}[Proof of Theorem \ref{MC 3dim 1}]
The proof is based on an inspection of the proofs of Theorem \ref{main theorem} in the case $n=2$, and of the proof of Theorem \ref{MC 3dim}. In the final part of the proof of Theorem \ref{main theorem} for $n=2$, we have seen that the argument of the proof works under the condition:
$$
\left(\frac{a_{22}}{f}\right)_{11}+\frac{a_{22}}f\ge0.
$$
Here, a local orthonormal coordinate system $\{e_1,e_2\}$ has been fixed on $\s^2$, and the subscript $1$ denotes the covariant differentiation along $e_1$. On the other hand, as we have seen in the proof of Theorem \ref{MC 3dim}, 
$$
a_{22}=w_{11}=h_{11}+h,
$$
where $h$ is the support function of the given convex body $\Omega_0$. Hence the condition becomes
$$
\left(\frac{h_{11}+h}{f}\right)_{11}+\frac{h_{11}+h}f\ge0.
$$
Note that the direction $e_1$ may be completely arbitrary. We conclude that in Theorem \ref{main theorem}, \eqref{cond-1 2} can be replaced by condition:  
\begin{equation}\label{new form dim two}
\left(\frac{h_{\alpha\alpha}(x)+h(x)}{f(x)}\right)_{\alpha\alpha}+\frac{h_{\alpha\alpha}(x)+h(x)}{f(x)}\ge0,\quad\forall\, x\in\s^2,\, \forall\, \alpha,
\end{equation}
where $\alpha$ is a tangential direction to $\s^2$ at $x$. 
\end{proof}

\section{The $n$-dimensional case}\label{section n general}

This section is devoted to the proof of Theorem \ref{main theorem}, when $n\ge3$.

\subsection{Some preliminary facts}\label{section preliminaries}
We assume throughout that $u$ is a solution to \eqref{the equation} in $D$, and that the assumptions of Theorem \ref{main theorem} concerning the tensor $(a_{\alpha\beta})_{\alpha,\beta=1,\dots,n}$ and the function $f$, are verified. The tensor $W=W(x)$ is defined in the statement of the theorem; as usual $w_{ij}$, $i,j=1,\dots,n$, denote the entries of $W$. By assumption \eqref{ellipticity} and standard regularity theory of elliptic equations, we have $u\in C^{4,\gamma}(\sfe)$, for some $\gamma\in(0,1)$. %{\color{red} [Is it correct?]}.

%In the subsequent part, we closely follow \cite{Guan-Ma-2003}, and in particular section 4.

\medskip

Given a square symmetric matrix $M\in\mbox{Sym}(n)$ and $j\in\{0,1,\dots,n\}$, let $\sigma_j(M)$ denote the $j$th elementary symmetric function of the eigenvalues of $M$. Note that, if $M$ is positive semi-definite, then the rank of $M$ is $l\in\{0,\dots,n-1\}$ if and only if
$$
\sigma_l(M)>0\quad\mbox{and}\quad\sigma_j(M)=0\quad\forall\, j=l+1,\dots,n.
$$
If 
$$
\sigma_n(W(x))>0\quad\forall\, x\in D,
$$
then $W$ has constant (and maximal) rank $n$ in $D$. So, in the perspective to prove Theorem \ref{main theorem}, we will assume that there exists an integer $l$ between $0$ and $(n-1)$, with the following property:
$$
\sigma_l(W(x))>0,\quad\forall\, x\in D,
$$
and there exists $x_0\in D$, such that
$$
\sigma_{l+1}(W(x_0))=0.
$$

\begin{remark}\label{l bigger than one} As already observed in the two-dimensional case, $l$ is at least $1$ by equation \eqref{the equation} and the strict positivity of $f$.
\end{remark}

We now introduce the {\it test function} $\phi$, which is one of the main tools of this proof. We remark that the choice of $\phi$ for $n\ge3$ is different from the one in the case $n=2$. Define
\begin{equation}\label{newphi}
\phi(x)=\sigma_{l+1}(W(x))+\frac{\sigma_{l+2}(W(x))}{\sigma_{l+1}(W(x))},
\end{equation}
for every $x$ such that $\sigma_{l+1}(W(x))>0$. 

\begin{remark}\label{testfunction} This test function was introduced in \cite{BG}, where the authors observed that, by the Newton-McLaurin inequalities, $\phi$ can be continuously extended to the case $\sigma_{l+1}(W)=0$, setting
\begin{equation*}
\frac{\sigma_{l+2}(W)}{\sigma_{l+1}(W)}=0\quad\mbox{if $\sigma_{l+1}(W)=0$.}
\end{equation*}
Therefore we complete the definition of $\phi$ setting
\begin{equation}\label{extension}
\phi(x)=0,\quad\mbox{$\forall\, x$ s.t. $\sigma_{l+1}(W(x))=0$.}
\end{equation} \end{remark}

%Denote\[\phi^{ij}=\frac{\partial \phi}{\partial w_{ij}}, \ \phi^{ij,km}=\frac{\partial^2 \phi}{\partial w_{ij}\partial w_{km}}\]

The general strategy of the proof will be to show that, under the assumptions of the theorem, $\phi$ verifies a differential inequality, to which the strong minimum principle can be applied. This will lead to the conclusion
$$
\phi\equiv0\quad\mbox{in $D$,}
$$
that is, $W$ has constant rank $l$ in $D$. 

\medskip

Let $D'$ be an open and connected set compactly contained in $D$, and let $C_0>0$ be such that 
\begin{equation}\label{lbound-l}
\sigma_l(W)\ge C_0\quad\mbox{in $D'$.}
\end{equation}
For any $z\in D'$, let $\lambda_1\le\lambda_2\le\dots\le\lambda_n$ be the eigenvalues of the matrix $W$ at $z$. As $u\in C^3(D)$, there exists a constant $C>0$, depending on $\|u\|_{C^3}$, $\|\phi\|_{C^2}$, $C_0$ and $n$, such that \[\lambda_{n}\ge\lambda_{n-1}\ge\dots\ge\lambda_{n-l+1}\ge C, \quad \mbox{in $D'$}.\] We set
$$ B=\{1,\dots,n-l\},\quad G=\{n-l+1,\dots,n\}. $$
%{\color{red} I think that here it should be $$G=\{1,\dots,l\},\quad B=\{l+1,\dots,n\}. $$}{\color{blue} We use $B=\{1,\dots,n-l\}$ so it's simpler for the expression of (\ref{control1}).} {\color{red} But then we have to reverse the order of the eigenvalues: $0\le\lambda_1\le\dots\le\lambda _n$; and the previous lower bound on the eigenvalues becomes $\lambda_n\ge\lambda_{n-1}\ge\dots\ge\lambda_{n-l}\ge C$. Correct?} {\color{blue} Correct}
We now recall a notation which is common in the context of constant rank type results. Given two functions $g$ and $h$ defined in $D'$, and $y\in D'$, we write
$$
g(y)\lesssim h(y)
$$
if there exist positive constants $c_1$ and $c_2$, such that
$$
g(y)-h(y)\le c_1|\nabla\phi(y)|+c_2\phi(y).
$$
We also write
$$
g(y)\sim h(y)
$$
if $g(y)\lesssim h(y)$ and $g(y)\lesssim h(y)$. Moreover, we write 
$$
g\lesssim h\qquad \mbox{in $D'$}
$$
if this inequality holds with constants $c_1$ and $c_2$ independent of $y$ and depending only on $\|u\|_{C^3}(D')$, $\|\phi\|_{C^2}(D')$, $C_0$ and $n$. 

As in the proof of Theorem \ref{main theorem} in the case $n=2$, we also write 
$$
g=O(\psi),
$$
for some positive function $\psi$ defined in $D'$, if 
$$
|g|\le C\psi\quad\mbox{in $D'$,}
$$
where, as above, the constant $C$ depends only on $\|u\|_{C^3}$, $\|\phi\|_{C^2}$, $C_0$ and $n$.

\medskip

In view of the Codazzi property of $W$ and arguments presented in \cite{BG}, $\phi$ has the properties listed below.

%{\color{red} [Question: from now on we have several fractions where the denominator is $\sigma_1(B)$ or a power of it. On the other hand $\sigma_1(B)$ can be zero. Is it clear that those fractions can be extended continuously to the case $\sigma_1(B)=0$?]} {\color{blue} I added Remark \ref{testfunction} to address this question.}

\begin{enumerate}
\item[{I.}] Once extended as indicated in \eqref{extension}, $\phi\in C^{1,1}$ (see Corollary 2.2 in \cite{BG}).
\item[{II.}] $\forall \alpha,\beta=1,\dots,n$, 
\begin{eqnarray}\label{new3-1}
\phi_{\alpha\beta}&=&O\left(\phi +\sum_{i,j\in
B}|\nabla
w_{ij}|\right)-\left[\frac{\sum_{i\in
B}V_{i\alpha}V_{i\beta}}{\sigma^3_{1}(B)} +\frac{\sum_{i,j\in B,
i\neq j}
w_{ij\alpha}w_{ji\beta}}{\sigma_{1}(B)}\right]+\nonumber \\
&&+\sum_{i\in B}\left[\sigma_l(G)+\frac{\sigma_{1}^{2}(B|i)+
\sigma_{2}(B|i)}{\sigma_{1}^{2}(B)}\right]
\left[w_{ii\alpha\beta} - 2\sum_{j\in
G}\frac{w_{ji\alpha}w_{ij\beta}}{w_{jj}}\right],
\end{eqnarray}
where
\[
V_{i\alpha}=\sigma_{1}(B)w_{ii\alpha}-w_{ii}\sum_{j\in B}w_{jj\alpha},
\] 
and $\sigma_1(B)$, $\sigma_k(B|i)$ denote respectively the $1$st, $k$th elementary symmetric function of the eigenvalues with indices in $B$, $B\setminus\{i\}$. 
%with $\lambda_i=0$.
Relation \eqref{new3-1} follows from Lemma 2.4 and Proposition 2.1 in \cite{BG} (see also the argument used to deduce (3.19) there). 
\item[{III.}] Let $m=|B|=n-l$. By Lemma 3.3 in \cite{BG}, there is a constant $C$ depending only on $n$, $\|W\|_{C^0}$ and $C_0$ in (\ref{lbound-l}), such that for each $\alpha=1,\dots,n$, for any $K>0$, $\delta>0$,
\begin{equation}\label{control1}
\sum_{i,j=1}^m|w_{ij\alpha}|\le C\left(1+\frac{2}{\delta}+K\right)
(\phi+|\nabla \phi|)+\frac{\delta}{2K}\frac{\sum_{i\neq j}^m|w_{ij\alpha}|^{2}}
{\sigma_{1}(B)}+
\frac{C}{K}\frac{\sum_{i=1}^mV_{i\alpha}^{2}}{\sigma_{1}^{3}(B)}.
\end{equation}
\end{enumerate}
\begin{remark}\label{newsim} By (\ref{new3-1}), (\ref{control1}) and the uniform ellipticity of $(a_{\alpha\beta})$, in the rest of the proof, we may write $w_{\alpha\beta i}\sim 0$ if $i\in B$ and at least one of $\alpha, \beta$ is in $B$. \end{remark} 

\subsection{The differential inequality for the function $\phi$}

In the sequel we will use the same notation adopted in the proof for the case $n=2$, for the covariant derivatives of the coefficients $a_{\alpha\beta}$. In particular,
$$
a_{\alpha\beta i},\quad a_{\alpha\beta ij}
$$
denote the first and second covariant derivatives of $a_{\alpha\beta}$, along the directions $e_i$, and along the directions $e_i$ and $e_j$ respectively, of a local orthonormal coordinate system on $\sfe$.

\begin{proposition} Let $D'$ be an open subset of $\mathbb{S}^n$, and let $u\in C^3(\Omega)$ be a solution of equation \eqref{the equation} in $D'$, such that the matrix $W$ defined in \eqref{W} is positive semi-definite in $D'$. Then $\phi$ verifies the following relation in $D'$ 
\begin{eqnarray}\label{equation seven}
\sum_{\alpha,\beta=1}^na_{\alpha\beta}\phi_{\alpha\beta}&\leq &\sum_{i\in B}\left[\sigma_l(G)+\frac{\sigma_{1}^{2}(B|i)+
\sigma_{2}(B|i)}{\sigma_{1}^{2}(B)}\right]\times\\ \nonumber
&&\times\left(-f+f_{ii}
-\sum_{\alpha\in G}a_{\alpha\alpha ii}w_{\alpha\alpha} -
2\sum_{\alpha,\beta\in G}\left[
a_{\alpha\beta i}w_{\alpha\beta i}+\sum_{j\in G}\frac{a_{\alpha\beta}w_{ij\alpha}w_{ij\beta}}{w_{jj}}
\right]\right) \\ \nonumber
&& +O(\phi+|\nabla \phi|).\nonumber
\end{eqnarray}
\end{proposition}

\begin{proof} For an arbitrary $z\in D'$, we choose a local orthonormal frame $\{e_1,\dots,e_n\}$ defined in a neighborhood of $z$, and we assume that $W(z)$ is diagonal, with the convention on sets $G$ and $B$ of ``good'' and ``bad'' subsets of indices of the eigenvalues of $W$ described before. We have 
$$
w_{ii}\sim0\quad\forall\, i\in B,\quad w_{ij}\sim0\quad\mbox{for $i\ne j$.}
$$
As a consequence of (\ref{new3-1})  and \eqref{covariant 2}, we obtain
%\begin{equation}\label{fi alfa beta}\frac1{\sigma_l}(\sigma_{l+1})_{\alpha\beta}\sim\sum_{i\in B}w_{\alpha\beta ii}-(n-l)w_{\alpha\beta}-2\sum_{i\in B, j\in G}\frac{w_{ij\alpha}w_{ij\beta}}{w_{jj}}.\end{equation}Then we have
\begin{eqnarray}\label{formula 3}
\sum_{\alpha,\beta=1}^na_{\alpha\beta}\phi_{\alpha\beta}&\leq &\sum_{i\in B}\left[\sigma_l(G)+\frac{\sigma_{1}^{2}(B|i)+
\sigma_{2}(B|i)}{\sigma_{1}^{2}(B)}\right]\left[\sum_{\alpha,\beta=1}^na_{\alpha\beta}w_{\alpha\beta ii}\right.\\ \nonumber
&&\left.-\sum_{\alpha,\beta=1}^na_{\alpha\beta}w_{\alpha\beta}-2\sum_{\alpha,\beta=1}^n\sum_{j\in G}a_{\alpha\beta}\frac{w_{ij\alpha}w_{ij\beta}}{w_{jj}}\right]\\ \nonumber && -\sum_{\alpha,\beta}\left[\frac{\sum_{i\in
B}a_{\alpha \beta}V_{i\alpha}V_{i\beta}}{\sigma^3_{1}(B)} +\frac{\sum_{i,j\in B,
i\neq j}a_{\alpha\beta}w_{ij\alpha}w_{ji\beta}}{\sigma_{1}(B)}\right]\\ \nonumber
&&+O\left(\phi +\sum_{i,j\in B}|\nabla w_{ij}|\right).
\end{eqnarray}

By the uniform ellipticity, there exists $\delta_0>0$ such that
\[
(a_{\alpha\beta}) \ge \delta_0 I_{n\times n},
\]
where $I_{n\times n}$ is the identity matrix of order $n$. Therefore
\[
\sum_{\alpha,\beta}\left[\frac{\sum_{i\in
B}a_{\alpha \beta}V_{i\alpha}V_{i\beta}}{\sigma^3_{1}(B)} +\frac{\sum_{i,j\in B,
i\neq j}a_{\alpha\beta}
w_{ij\alpha}w_{ji\beta}}{\sigma_{1}(B)}\right]\ge \delta_0 \left[\frac{\sum_{i\in
B}V^2_{i\alpha}}{\sigma^3_{1}(B)} +\frac{\sum_{i,j\in B,
i\neq j}
|w_{ij\alpha}|^2}{\sigma_{1}(B)}\right]
\]
Therefore, by (\ref{formula 3}),
\begin{eqnarray}\label{formula 4}
\sum_{\alpha,\beta=1}^na_{\alpha\beta}\phi_{\alpha\beta}&\leq &\sum_{i\in B}\left[\sigma_l(G)+\frac{\sigma_{1}^{2}(B|i)+
\sigma_{2}(B|i)}{\sigma_{1}^{2}(B)}\right]\left[\sum_{\alpha,\beta=1}^na_{\alpha\beta}w_{\alpha\beta ii}\right.\\ \nonumber
&&\left.-\sum_{\alpha,\beta=1}^na_{\alpha\beta}w_{\alpha\beta}-2\sum_{\alpha,\beta=1}^n\sum_{j\in G}a_{\alpha\beta}\frac{w_{ij\alpha}w_{ij\beta}}{w_{jj}}\right]\\ \nonumber && -\delta_0 \left[\frac{\sum_{i\in
B}V^2_{i\alpha}}{\sigma^3_{1}(B)} +\frac{\sum_{i,j\in B,
i\neq j}
|w_{ij\alpha}|^2}{\sigma_{1}(B)}\right]+O\left(\phi +\sum_{i,j\in B}|\nabla w_{ij}|\right).
\end{eqnarray}

Next, we use (\ref{control1}) to deduce that
$$
\sum_{i,j\in B}|\nabla w_{ij}|
$$ 
can be controlled by 
\[
\delta_0 \left[\frac{\sum_{i\in
B}V^2_{i\alpha}}{\sigma^3_{1}(B)} +\frac{\sum_{i,j\in B,
i\neq j}
|w_{ij\alpha}|^2}{\sigma_{1}(B)}\right] +O(\phi +|\nabla \phi|). 
\] 
Hence
\begin{eqnarray}\label{formula 30}
\sum_{\alpha,\beta=1}^na_{\alpha\beta}\phi_{\alpha\beta}&\le &\sum_{i\in B}\left[\sigma_l(G)+\frac{\sigma_{1}^{2}(B|i)+
\sigma_{2}(B|i)}{\sigma_{1}^{2}(B)}\right]\left[\sum_{\alpha,\beta=1}^na_{\alpha\beta}w_{\alpha\beta ii}\right.\\ \nonumber
&&\left.-\sum_{\alpha,\beta=1}^na_{\alpha\beta}w_{\alpha\beta}-2\sum_{\alpha,\beta=1}^n\sum_{j\in G}a_{\alpha\beta}\frac{w_{ij\alpha}w_{ij\beta}}{w_{jj}}\right]\\ \nonumber
&& +O(\phi +|\nabla \phi|).
\end{eqnarray}
Next, for $i\in B$, we differentiate twice \eqref{the equation} with respect to $e_i$: 
$$
\sum_{\alpha,\beta=1}^n\left[
a_{\alpha\beta ii}w_{\alpha\beta}+2a_{\alpha\beta i}w_{\alpha\beta i}+a_{\alpha\beta}w_{\alpha\beta ii}\right]=f_{ii};
$$
whence
\begin{eqnarray*}
\sum_{\alpha,\beta=1}^n\sum_{i\in B}a_{\alpha\beta}w_{\alpha\beta ii}&=&\sum_{i\in B}f_{ii}-
\sum_{\alpha,\beta=1}^n\sum_{i\in B}a_{\alpha\beta ii}w_{\alpha\beta}-2\sum_{\alpha,\beta=1}^n\sum_{i\in B}a_{\alpha\beta i}w_{\alpha\beta i}\\
&\sim&\sum_{i\in B}f_{ii}-\sum_{\alpha\in G}^n\sum_{i\in B}a_{\alpha\alpha ii}w_{\alpha\alpha}-2\sum_{\alpha,\beta=1}^n\sum_{i\in B}a_{\alpha\beta i}w_{\alpha\beta i}.
\end{eqnarray*}
Replacing this relation into \eqref{formula 30} we obtain (\ref{equation seven}).
\end{proof}

We will need a further auxiliary result. Let $\M_{l\times l}$ be the space of $l\times l$ matrices, where $l\in\N$. %We define, for $M_1, M_2\in\M_{l\times l}$, 
We consider the standard Euclidean inner product for matrices defined by
$$
\langle M_1,M_2\rangle=\mbox{tr}(M_1 M_2^T),
$$
where $M_1, M_2\in\M_{l\times l}$, and the apex $T$ denotes the transpose. %It is easy to check that this is a scalar product in $\M_{l\times l}$.

\begin{lemma}\label{rank-reduction}  
Let $l\ge2$ be an integer number and let $R,S,T\in\M_{l\times l}$. Assume in particular that $R$ is symmetric and positive definite. Let 
$$
{\mathcal D}=\{W\in\symm(l)\colon\langle R,W\rangle=1, W\ge0\}.
$$
Consider the function $\rho\colon{\mathcal D}\to\R$ defined by
\[
\rho(W)=\langle S, W\rangle+\langle T,W\rangle^2.
\]
Then there exists $W_{\rm min}\in\mathcal D$, with $rank(W_{\rm min})=1$, such that 
$$
\rho(W_{\rm min})=\min_{\mathcal D}\rho.
$$
\end{lemma}

%[Andrea: I removed ``only'' in the last sentence of the Lemma. I don't think we prove it, and probably we don't need it, but I might be wrong.]

%{\color{red}{[Pengfei: Andrea is right. Note that the extra term $(\frac{a'_{11}}{a_{11}})^2-\sum_{i,j=1}^l \frac{a^{i,j}a'_{1i}a'_{1j}}{a_{11}}$ occurs only when we push the minimum point to the relative boundary. Since we assume rank of $W$ is $l$, $W$ can not get too close to the boundary. This suggests that when we replace $G_l(W_m,W'_m)$ by $G_l(W_m, X_m)$, we waste some capital!}]}

\begin{proof} Note that $\mathcal D$ is convex and compact (the second claim follows being $R$ positive definite), and that $\langle R,R\rangle>0$. Let
$$
S'=S-\frac{{\langle S,R\rangle}}{\langle R,R\rangle}R,\quad 
T'=T-\frac{{\langle T,R\rangle}}{\langle R,R\rangle}R.
$$
Clearly 
$$
\langle S',R\rangle=\langle T',R\rangle=0.
$$
Now, using that $\langle R,W\rangle=1$, $\rho(W)$ rewrites for $W\in{\mathcal D}$ as
\begin{eqnarray*}
\rho(W)&=&\langle S',W\rangle+\frac{\langle S,R\rangle}{\langle R,R\rangle}+\left[
\langle T',W\rangle+\frac{\langle T,R\rangle}{\langle R,R\rangle}\right]^2\\
&=&\langle S',W\rangle+a+\left[\langle T',W\rangle+b\right]^2,
\end{eqnarray*}
where $a$ and $b$ are constants. Hence, it suffices to prove the lemma for the function $\bar\rho$ defined by
$$
\bar\rho(W)=\langle S',W\rangle+\left[\langle T',W\rangle+b\right]^2,\quad W\in{\mathcal D}.
$$
As $\mathcal D$ is compact, there exists $W_0\in{\mathcal D}$, such that
\[
\bar\rho(W_0)=\min_{W\in\mathcal {D}}\bar\rho(W).
\]
Let $k$ be the rank of $W_0$, so $1\le k\le l$. If $k=1$, then the conclusion of the lemma is true. Hence consider the case $k\ge 2$. After a rotation, we may assume that the entries $w_{ij}$ of $W_0$ verify $w_{ij}=0$ for every $i>k$. We may also embed $\symm(k)$ in $\symm(l)$, with the convention that
\[
\symm(k)=\{W=(w_{ij})_{i,j=1,\dots,l}\in \symm(l) | w_{ij}=0, \forall i>k\}.
\]
Define
\[
{\mathcal D}_k=\{W\in \symm(k)\colon\langle R,W\rangle=1, W\ge0\}.
\]
Note that ${\mathcal D}_k\subset{\mathcal D}$ is convex and compact. Let $\mathcal H$ be the subspace of $\symm(k)$ defined by
$$
{\mathcal H}=\{Y\in \symm(k)\colon \langle Y,R\rangle=0\}.
$$
In particular, $T'\in {\mathcal H}$; moreover, ${\mathcal H}$ has dimension 
$$
\frac{k(k+1)}{2}-1\ge2,
$$
as $k\ge2$. Hence there exists $\bar T\in{\mathcal H}$ such that 
$$
\langle T',\bar T\rangle=0,\quad \langle \bar T,\bar T\rangle=1.
$$
%Let $\bar S=\langle S',\bar T\rangle \bar T$ be the component of $S'$ parallel to $\bar T$. Let $W_0$ be a point in $\mathcal D$ where $\bar\rho$ attains its minimum.  
As $\bar T\in{\mathcal H}$, the points of the form $W=W_0+x\bar T$, for $x\in\R$, verify $\langle W,R\rangle =1$. Let
$$
L=\{W\in \symm(k)\colon W=W_0+x\bar T\}\quad\mbox{and}\quad \Sigma=L\cap{\mathcal D}.
$$
$L$ is a straight line in $\symm(l)$, lying on the same hyperplane as ${\mathcal D}_k$, and by the convexity of ${\mathcal D}_k$, $\Sigma$ is a segment containing $W_0$ in its interior (because the rank of $W_0$ is $k$), that is
$$
\Sigma=\{W_0+x\bar T\colon x\in[-c,d]\}
$$
for some $c,d>0$. In particular $W_0-c\bar T$ and $W_0+d\bar T$ belong to the relative boundary of ${\mathcal D}_k$ (meaning that they are points of ${\mathcal D}_k$ not belonging to the relative interior of ${\mathcal D}_k$). Consider the function $\gamma\colon[-c,d]\to\R$ defined by
\begin{eqnarray*}
\gamma(x)&=&\bar\rho(W_0+x\bar T)=\langle W_0+x\bar T, S'\rangle+\left[\langle W_0+x\bar T,T'\rangle +b\right]^2\\
&=&\langle W_0+x\bar T, S'\rangle+\left[\langle W_0,T'\rangle +b\right]^2.
\end{eqnarray*}
Thus, $\gamma$ is a linear function with a minimum point at $x=0$. Therefore $\gamma$ is constant, which shows that the minimum of $\bar\rho$ is attained (also) on the relative boundary of ${\mathcal D}_k$. In other words, there exists $\tilde W_0\in {\mathcal D}_{k-1}$ such that
\[
\bar\rho(\tilde W_0)=\min_{W\in {\mathcal D}}\bar\rho(W).
\]
Iterating the previous argument, we deduce that the minimum is attained on matrices with rank one.
\end{proof}

\begin{proof}[Proof of Theorem \ref{main theorem}.] As $f>0$, we may assume that $f\equiv1$ on $D$, up to passing from the tensor $A$ to 
$$
\frac{A}f.
$$

\medskip

Let $l\in\{1,\dots,n-1\}$ be the minimal rank of the matrix $W$ on $D$, attained at some point $x_0\in D$. Now fix any point $z$ in a sufficiently small neighborhood $D'$ of $x_0$ and call $\ell$ the rank of $W$ at $z$. We have $\ell\geq l$. We will consider $W=W(z)$ as an $\ell\times\ell$ matrix, with the convention that $w_{ij}=0$, for every $i>\ell$. We set 
$$
(w^{jk})_{j,k=1,\dots,\ell}=W^{-1},
$$
if no ambiguity may occur. The next calculation will be performed at the point $z$.

We will prove that, in $D'$, for every $i\in B$, 
\begin{equation}\label{formula 5}
\sum_{\alpha,\beta=1}^{\ell} a_{\alpha\beta ii}w_{\alpha\beta}+2\sum_{\alpha,\beta=1}^\ell\left(
a_{\alpha\beta i}w_{\alpha\beta i}+\sum_{j,k=1}^{\ell} a_{\alpha\beta}w^{jk}w_{\alpha ki}w_{\beta ji}
\right)\gtrsim -1.
\end{equation}
By (\ref{equation seven}), with $f\equiv 1$, the previous inequality implies that 
\begin{equation}\label{added 10}
\sum_{\alpha,\beta=1}^na_{\alpha\beta}\phi_{\alpha\beta}\lesssim0.
\end{equation}
As a consequence of \eqref{added 10} and the strong minimum principle, $\phi\equiv0$ in $D$, and the first conclusion of Theorem \ref{main theorem} will follow.

\bigskip

Let us fix $i\in B=\{\ell+1,\dots,n\}$, and set
$$
\Gamma_i:=\sum_{\alpha,\beta=1}^\ell a_{\alpha\beta ii}w_{\alpha\beta}+2\sum_{\alpha,\beta=1}^\ell
\left(a_{\alpha\beta i}w_{\alpha\beta i}+\sum_{j,k=1}^\ell a_{\alpha\beta}w^{jk}w_{\alpha ki}w_{\beta ji}
\right).
$$
In the rest of the proof, for simplicity we use $=$ for $\sim$. The following relations hold:
\begin{equation}\label{constraints}
\sum_{\alpha,\beta=1}^\ell a_{\alpha\beta} w_{\alpha\beta}=1,\quad\sum_{\alpha,\beta=1}^\ell a_{\alpha\beta i}w_{\alpha\beta}= -\sum_{\alpha,\beta=1}^\ell a_{\alpha\beta}w_{\alpha\beta i}.
\end{equation}
Here we used the fact that, by Remark \ref{newsim}, $w_{\alpha\beta i}\sim 0$ if at least one of $\alpha, \beta$ is in $B$. We set:
\begin{eqnarray*}
A=(a_{\alpha\beta})_{\alpha,\beta=1}^\ell,\; W=(w_{\alpha\beta})_{\alpha,\beta=1}^\ell,\; A'=(a_{\alpha\beta i})_{\alpha,\beta=1}^\ell,\;A''=(a_{\alpha\beta ii})_{\alpha,\beta=1}^\ell,\; W'=(w_{\alpha\beta i})_{\alpha,\beta=1}^\ell;
\end{eqnarray*}
note that $W'$ is symmetric. Then $\Gamma_i$ can be written in the form
\begin{equation}\label{bilinear}
\Gamma_i(W,W')=\langle A'', W\rangle+2\left[\langle A',W'\rangle+\langle W',\Theta W'\rangle\right],
\end{equation}
where
$$
\Theta W'=W^{-1}W' A,
$$
and the equalities in \eqref{constraints} become
\begin{equation}\label{dubbio}
\langle A,W\rangle=1,\quad \langle A',W\rangle=-\langle A,W'\rangle.
\end{equation}

In the above constraints, $=$ should be intended as $\sim$, in the sense that the right hand side of the first one should be $1+a(z)$, with $a(z)\sim 0$, and in the right hand side of the second one we should add a term $b(z)$, with $b(z)\sim 0$. However, notice that this will not affect the conclusion, which is that we can calculate the minimum of a certain functional $\tilde\Gamma_i$ over a certain subset $\mathcal C$ of $\symm(\ell)\times\symm(\ell)$ at a matrix of rank 1, so that we get $\Gamma_i\gtrsim -1$, that is, \eqref{formula 5}. The set $\mathcal C$ should indeed depend on $z$, via $a(z)$ and $b(z)$, but the conclusion does not change, for every fixed $z$, as desired.

Now, with a slight abuse of notation, we consider two fixed arbitrary matrices $A$ and $A'$ in $\M_{\ell\times\ell}$, such that $A$ is symmetric and positive definite, and we define
\begin{equation}\label{constraint}
{\mathcal C}=\{(W,Y)\in\symm(\ell)\times\symm(\ell)\colon W\ge0,\ \langle A,W\rangle=1,\ \langle A',W\rangle=-\langle A,Y\rangle\}.
\end{equation}
Note that if $(W,Y)\in{\mathcal C}$ and $W$ is positive definite, then $(W,Y)$ belongs to the relative interior of $\mathcal C$. We consider now the function $\hat\Gamma_i\colon{\mathcal C}\to\R$, defined by:
\begin{equation}\label{bilinear2}
\hat\Gamma_i(W,Y)=\langle A'', W\rangle+2\left[\langle A',Y\rangle+\langle Y,\Theta Y\rangle\right],
\end{equation}
where $\Theta\colon\symm(\ell)\to\symm(\ell)$ is the linear map defined by
$$
\Theta P=W^{-1}PA,\quad\forall\, P\in\symm(\ell).
$$
The scope of this part of the proof is to minimize $\hat \Gamma_i$ over $\mathcal C$. If we identify $\M_{\ell\times\ell}$ with $\R^{\ell^2}$, $\langle P,Q\rangle$ is the canonical inner product in $\R^{\ell^2}$. Adopting this point of view, $\Theta$ can be seen as a positive definite matrix, to which we associate a linear operator from $\R^{\ell^2}$ to $\R^{\ell^2}$, and
\begin{equation}\label{theta -1}
\Theta^{-1}Q=WQA^{-1},\quad\forall\ Q\in\R^{\ell^2}.
\end{equation}
%We aim to prove that, if $l\ge2$, then the minimum of $G_\ell$ on $\mathcal C$ can not be attained in the relative interior of $\mathcal C$. 
Let $(W_m, Y_m)$, $m\in\N$, be a minimizing sequence for $\hat \Gamma_i$ on $\mathcal C$, such that $W_m>0$ for every $m$. As $\Theta$ depends on $W$, we denote by $\Theta_m$ the operator corresponding to $W_m$.

It follows from the expression of $\hat \Gamma_i$ that, for each fixed $m$, the function
$$
X\rightarrow\hat \Gamma_i (W_m,X),
$$
defined on the set $\{X\in\M_{\ell\times\ell}\colon (W_m,X)\in \mathcal{C}\}$, 
has a unique minimum point $X_m$. 
%In particular, {\color{green}by definition of $\mathcal{C}$}
%$$
%\langle A',W_m\rangle+\langle A,X_m\rangle=0.
%$$
By the Lagrange multipliers theorem, there exists $\lambda\in\R$ such that
$$
A'+2\Theta_m X_m=\lambda A\quad\Rightarrow\quad2\Theta_m X_m=\lambda A-A'.
$$
Therefore, using the definition of $\Theta^{-1}$,
\begin{eqnarray*}
X_m&=&\frac12\Theta_m^{-1}(\lambda A-A')\\
&=&\frac12W_m(\lambda A-A')A^{-1}\\
&=&\frac12\lambda W_m-\frac12W_mA'A^{-1}.
\end{eqnarray*}
As, by the definition of $\mathcal{C}$,
$$
\langle A,X_m\rangle=-\langle W_m,A'\rangle\quad\mbox{and}\quad\langle A,W_m\rangle =1\quad\forall\, m,
$$
we obtain
$$
-\langle W_m,A'\rangle=\frac12\lambda-\frac12\langle A,W_mA'A^{-1}\rangle.
$$
On the other hand, since $A$ and $W_m$ are symmetric, we find
$$
\langle A,W_mA'A^{-1}\rangle=\mbox{tr}(AW_mA'A^{-1})=\mbox{tr}(W_m A')=\langle W_m,A'\rangle.
$$
Therefore, $\lambda=-\langle W_m,A'\rangle$ so that
%$$ X_m={\color{green}-}\frac12\langle W_m, A'\rangle W_m-\frac12W_mAA'={\color{green}-\frac12\langle W_m, A'\rangle W_m-\frac12\Theta_m^{-1}A'}.
%$${\color{red} [Andrea: the last line should be in fact
$$
X_m=-\frac12\langle W_m, A'\rangle W_m-\frac12W_mA'A^{-1}=-\frac12\langle W_m, A'\rangle W_m-\frac12\Theta_m^{-1}A'.
$$%]}

%{\color{red} [Andrea: note that the first of the above relation (and consequently the second as well) is different from the notes sent by Pengfei - please check.]}

%{\color{red} [Pengfei: Andrea is right, there was a calculation error in my note. The rest is essentially the same with adjustment of a factor of $2$ as Andrea wrote in (20).]}

By \eqref{bilinear2},
\begin{equation}\label{bilinear 2}
\hat\Gamma_i(W_m,X_m)=\langle A'',W_m\rangle+2\bar \Gamma_i(W_m,X_m)
\end{equation} 
where
$$
\bar \Gamma_i(W_m,X_m)=\langle A',X_m\rangle+\langle X_m,\Theta X_m\rangle.
$$
We compute now:
\begin{eqnarray}\label{e:Gammahat WmXm}
\bar \Gamma_i(W_m,X_m)&=&-\frac12\langle A',\langle W_m,A'\rangle W_m\rangle-\frac12\langle A',\Theta_m^{-1}A'\rangle\notag\\
&&+\frac14\langle\langle W_m,A'\rangle W_m+\Theta_m^{-1}A',\Theta_m\left(\langle W_m, A'\rangle W_m+\Theta_m^{-1}A'\right)\rangle\notag\\
&=&-\frac12\langle A',W_m\rangle^2-\frac12\langle A',\Theta_m^{-1}A'\rangle
+\frac14\langle\langle W_m,A'\rangle W_m+\Theta_m^{-1}A',\langle W_m, A'\rangle\Theta_m W_m+A'\rangle\notag\\
&=&-\frac12\langle A',W_m\rangle^2-\frac12\langle A',\Theta_m^{-1}A'\rangle
+\frac14\langle\langle W_m,A'\rangle W_m+\Theta_m^{-1}A',\langle W_m, A'\rangle A+A'\rangle\notag\\
&=&-\frac12\langle A',W_m\rangle^2-\frac12\langle A',\Theta_m^{-1}A'\rangle\notag\\
&&+\frac14(\langle A',W_m\rangle^2\langle A,W_m\rangle+\langle A',W_m\rangle^2
+\langle\Theta_m^{-1}A',A\rangle\langle A',W_m\rangle+\langle\Theta_m^{-1}A',A'\rangle)\notag\\
&=&\frac14\langle A',W_m\rangle^2-\frac14\langle A',\Theta_m^{-1}A'\rangle\,
\end{eqnarray}
where in the last equality we have used that $\langle W_m,A\rangle=1$ and that $\langle\Theta_m^{-1}A',A\rangle=\langle W_m,A'\rangle$. Define
\begin{eqnarray}\label{bilinear 3}
\tilde\Gamma_i(W)&=&\langle A'',W\rangle+\frac12\langle A',W\rangle^2
-\frac12\langle A',\Theta^{-1}A'\rangle\\
&=&\langle A'',W\rangle+\frac12\langle A',W\rangle^2-\frac12\langle A', WA'A^{-1}\rangle.
\end{eqnarray} 
Note that, by \eqref{bilinear 2} and \eqref{e:Gammahat WmXm} and the definition of $X_m$, we have $\tilde \Gamma_i(W_m)=\hat \Gamma_i(W_m, X_m)\le \hat \Gamma_i(W_m, Y)$, for every $Y\in \{X\in\M_{\ell\times\ell}\colon (W_m,X)\in \mathcal{C}\}$. We want to find the minimum of $\tilde \Gamma_i(W)$ under the constraint
\[
\langle A,W\rangle=1, \ W\ge 0.
\]
If $\ell\ge2$, by Lemma \ref{rank-reduction}, the minimum of $\tilde \Gamma_i$ is attained at some $W\ge 0$ with rank $1$. After a proper rotation, we may assume $W$ is diagonal with $w_{11}>0$ and $w_{jj}=0$ for every $j\ge 2$. From the constraint $\langle W, A\rangle =1$, we deduce
\[
w_{11}=\frac{1}{a_{11}}\,,
\]
whence
\begin{equation}\label{cond-ll} 
\min_{\langle A,W\rangle=1, \ W\ge 0}\tilde \Gamma_i(W)=\frac{a_{11}''}{a_{11}}+\frac12\left(\frac{a'_{11}}{a_{11}}\right)^2-\frac12\sum_{i,j=1}^\ell \frac{a^{ij}a'_{1i}a'_{1j}}{a_{11}},
\end{equation}
where
$$
(a^{ij})_{i,j=1,\dots,n}=A^{-1}.
$$
Therefore, by the assumption in \eqref{cond-l} we conclude \eqref{formula 5}.

On the other hand, if $\ell=1$, then, obviously, as before the minimum of $\tilde \Gamma_1$ is attained at some $W\ge 0$ with rank $1$. After a proper rotation, we may assume $W$ is diagonal with $w_{11}>0$. Hence
\begin{equation}\label{cond-ll-bis} 
\min_{\langle A,W\rangle=1, \ W\ge 0}\tilde \Gamma_i(W)=\frac{a_{11}''}{a_{11}}+\frac12\left(\frac{a'_{11}}{a_{11}}\right)^2-\frac12\sum_{i,j=1}^1 \frac{a^{ij}a'_{1i}a'_{1j}}{a_{11}},
\end{equation}
and we conclude \eqref{formula 5} as above by means of \eqref{cond-l}. In particular, note that in this case
$$
\left(\frac{a'_{11}}{a_{11}}\right)^2-\sum_{i,j=1}^1 \frac{a^{ij}a'_{1i}a'_{1j}}{a_{11}}=0.
$$
This concludes the proof in the case of a general subset $D\subset\sfe$.

\medskip

We now consider the case $D=\sfe$ and we show that constant rank of $W$ on $\mathbb S^n$ implies that $W$ has full rank, i.e., $W>0$ on $\sfe$. To this aim, we assume by contradiction that the (constant) rank is $l<n$; we know that $l\ge 1$. Thus, 
\[
\sigma_{l}(W(x))=:q(x)>0, \quad \sigma_{l+1}(W(x))=0, \quad \forall x\in \mathbb S^n.
\] 
Let $\delta_0>0$ be such that
\[
q(x)\ge \delta_0, \quad \forall x\in \mathbb S^n.
\]
Let $\Omega$ be the convex body with support function $u$. We may assume $u\ge 0$ by moving the Steiner point of $\Omega$ to the origin. By Minkowski identities (see \cite[Section 5.3.1]{Schneider}), there exists $C_{n,l}>0$, such that
\begin{eqnarray*} 0&=&  C_{n,l} \int_{\mathbb S^n}\sigma_{l+1}(W(x))dx \\
&=& \int_{\mathbb S^n}u\sigma_{l}(W(x))dx\\
&\ge & \int_{\mathbb S^n}u q(x)dx\ge \delta_0 \int_{\mathbb S^n}udx. 
\end{eqnarray*}
As $u\ge 0$, we deduce $u\equiv 0$, that is, a contradiction.
\end{proof}

\begin{remark} %\color{red}{This is puzzling.} 
Let $A(x)=(a_{ij}(x))\in C^{2,\gamma}(\s^n)$ be a positive $(2,0)$-tensor, for some $\alpha\in(0,1)$, and let $h\in C^{2,\alpha}(\sfe)$, $h>0$. Then the quantity
\begin{equation*}
\left(\frac{a'_{11}}{a_{11}}\right)^2-\sum_{i,j=1}^l \frac{a^{i,j}a'_{1i}a'_{1j}}{a_{11}}
\end{equation*} 
is invariant if we pass from $A$ to $hA$.  
In particular if the $1$-homogeneous extension of $A$ to $\R^n$ is convex, and if
\begin{equation*}
\left(\frac{a'_{11}}{a_{11}}\right)^2-\sum_{i,j=1}^l \frac{a^{i,j}a'_{1i}a'_{1j}}{a_{11}}\le0\quad\mbox{on $\sfe$,}
\end{equation*} 
then the assumptions of Theorem \ref{main theorem} are verified by $hA$, for every $h>0$.
\end{remark}

\begin{remark} In the notations of the proof of Theorem \ref{main theorem}, if $n=2$ and $l=1$, then  
\begin{equation}
\left(\frac{a'_{11}}{a_{11}}\right)^2-\sum_{i,j=1}^1 \frac{a^{ij}a'_{1i}a'_{1j}}{a_{11}}=0.
\end{equation}
In general, 
\begin{equation}
\left(\frac{a'_{11}}{a_{11}}\right)^2-\sum_{i,j=1}^l \frac{a^{ij}a'_{1i}a'_{1j}}{a_{11}}=0
\end{equation} 
if $A'=\lambda A$ for some $\lambda\in \mathbb R$.

\end{remark}
\begin{remark} 
In general, for $l\ge 2$, if $WA=AW, \ W>0$, 
\[\langle A',W\rangle^2-\langle A', WA'A^{-1}\rangle\le 0,\] with equality holds if and only if $A'=\lambda A$ for some $\lambda\in \mathbb R$. To see this, we first observe that for any matrix $B$, 
\begin{equation}\label{obT}
\langle B, B\rangle=\langle B, B^T\rangle 
\end{equation} 
where $B^T$ the transpose of $B$. Write 
\[A=A^{\frac12}A^{\frac12}, \ W=W^{\frac12}W^{\frac12}.\] As $\langle A, W\rangle =1$, by (\ref{obT})
\[\langle A^{\frac12}W^{\frac12}, A^{\frac12}W^{\frac12}\rangle =1.\]We also have  
\[\langle A',WA'A^{-1}\rangle =\langle W^{\frac12}A'A^{-\frac12}, W^{\frac12}A'A^{-\frac12}\rangle,\] and
\[\langle A',W\rangle =\langle A^{\frac12}W^{\frac12}, W^{\frac12}A'A^{-\frac12}\rangle.\]

By Cauchy-Schwarz inequality,
\[\langle A',W\rangle^2\le \langle A^{\frac12}W^{\frac12}, A^{\frac12}W^{\frac12}\rangle \langle W^{\frac12}A'A^{-\frac12}, W^{\frac12}A'A^{-\frac12}\rangle=\langle W^{\frac12}A'A^{-\frac12}, W^{\frac12}A'A^{-\frac12}\rangle, \] with equality holds if and only if 
\[W^{\frac12}A'A^{-\frac12}=\lambda A^{\frac12}W^{\frac12} \] for some $\lambda\in \mathbb R$. As $WA=AW$, we get $A'=\lambda A$. This is satisfied if $l=1$.
\end{remark}

\section{Proof of Corollaries \ref{the corollary} and \ref{the corollary nge3}}\label{final section}
As stated in the introduction, we exploit a continuity argument. Let $f$ and $A$ be as in the statement of either Corollary \ref{the corollary} or Corollary \ref{the corollary nge3}. Let
$$
A_0\equiv I_n,\quad A_1=A,
$$ 
where $I_n$ is the identity matrix. Moreover, for $t\in(0,1)$, let
$$
A_t=(1-t)A_0+tA_1.
$$
Note that, for every $t\in[0,1]$, $A_t$ and $f_t$ verify the assumptions of Theorem \ref{main theorem}. For $n\ge3$ this is directly assumed in Corollary \ref{the corollary nge3}. For $n=2$ this is due to the fact that condition \eqref{cond-1 2} is linear with respect to $A$, and the condition is verified by the pair $f$, $A$, and by the pair $f$, $I_n$ (here the assumption that condition \eqref{GM condition} holds for $f$ is used -- see Remark \ref{remark}). 

For $t\in[0,1]$ let $u_t$ be the solution of equation \eqref{the equation}, with $A=A_t$ and $f=f_t$, and let 
\[
W_t=((u_t)_{ij}+u_t\delta_{ij})_{i,j=1,\dots,n}\ge 0,\quad t\in[0,1].
\] 
We consider the set
$$
{\mathcal T}=\{t\in[0,1]\colon W_t>0\,\mbox{on $\sfe$}\}.
$$
By Theorem \ref{GM thm}, $0\in{\mathcal T}$, and in particular $\mathcal T\ne\emptyset$. Moreover, the condition $W_t>0$ implies that  $\mathcal T$ is open. On the other hand, Theorem \ref{main theorem} ensures that $\mathcal T$ is closed. We deduce that ${\mathcal T}=[0,1]$, and this concludes the proof.

\end{document}